\documentclass[a4paper,10pt]{article}
\usepackage[utf8x]{inputenc}
\usepackage{tracefnt,amsmath,tabu,array}
\usepackage{amssymb,graphicx,setspace,amsfonts,amsbsy}
\usepackage{pifont,latexsym,ifthen,amsthm,rotating,calc,textcase,booktabs,cancel,slashed}

\newtheorem{theorem}{Theorem}[section]
\newtheorem{lemma}[theorem]{Lemma}
\newtheorem{corollary}[theorem]{Corollary}
\newtheorem{proposition}[theorem]{Proposition}
\newtheorem{definition}[theorem]{Definition}

\newtheorem{remark}[theorem]{Remark}
\newcommand{\filledbox}{\leavevmode
  \hbox to.77778em{%
  \hfil\vbox to.675em{\hrule width.6em height.6em}\hfil}}

\newcommand{\Rm}{{\mathbb R}}

\begin{document}
%\doublespacing
\tabulinesep=1.0mm
% Enter full title and short title for running headers
\title{Asymptotic energy distribution of one-dimensional nonlinear wave equation\footnote{MSC classes: 35L05, 35L71.}}
%This work is supported by National Natural Science Foundation of China Programs 11601374, 11771325}}

\author{Ruipeng Shen\\
Centre for Applied Mathematics\\
Tianjin University\\
Tianjin, China}

\maketitle

\begin{abstract}
 In this work we consider the defocusing nonlinear wave equation in one-dimensional space. We show that almost all energy is located near the light cone $|x|=|t|$ as time tends to infinity. We also prove that any light cone will eventually contain some energy. As an application we obtain a result about the asymptotic behaviour of solutions to focusing one-dimensional wave equation with compact-supported initial data. 
\end{abstract}

\section{Introduction}
 In this work we consider the defocusing one-dimensional wave equation ($p>1$)
 \begin{equation} \label{cp1}
  \left\{\begin{array}{ll} u_{tt} - u_{xx} = - |u|^{p-1}u, & (t,x) \in \Rm\times \Rm; \\
  u(x,0) = u_0(x), & x\in \Rm; \\ u_t(x,0) = u_1(x), & x\in \Rm.  \end{array}\right.
 \end{equation}
 Throughout this paper we always assume $(u_0,u_1) \in (\dot{H}^1\cap L^{p+1}) \times L^2$, thus the solution comes with finite energy and momentum, which are the most important conserved quantities of this equation. 
\begin{align*}
 E(u,u_t) & = \int_{-\infty}^\infty \left(\frac{1}{2}|u_x(x,t)|^2 + \frac{1}{2}|u_t(x,t)|^2 + \frac{1}{p+1}|u(x,t)|^{p+1}\right) dx = \hbox{Const};  \\
 M(u,u_t) & = \int_{-\infty}^\infty u_x(x,t) u_t(x,t) dx = \hbox{Const}.
\end{align*}

\paragraph{Local well-posedness} The existence and uniqueness of local solutions follows a similar argument as given in Ginibre-Soffer-Velo \cite{locad1} and Lindblad-Sogge \cite{ls}. For the reader's convenience we stretch a proof as below. First of all, the embedding $\dot{H}^1\cap L^{p+1}(\Rm) \hookrightarrow C(\Rm)$ and a Gagliardo-Nirenberg inequality 
\[
 f(x_0) \leq \left(\int_{x_0}^{+\infty} |f'(x)|^2 dx \right)^{\frac{1}{p+3}} \left(\int_{x_0}^{+\infty} |f(x)|^{p+1}\right)^{\frac{1}{p+3}} 
\]
imply that $u_0 \in C(\Rm_x)$ is bounded and converges to zero as $x\rightarrow \infty$. We define a constant 
\[
 A =  \sup_{(x,t) \in \Rm \times [0,1]} \left| u_0(x-t)+u_0(x+t) + \int_{x-t}^{x+t} u_1(x') dx'\right| < +\infty.
\]
Next we introduce a complete space
\begin{align*}
 &X =  \left\{u\in C(\Rm_x \times [0,T]): \lim_{x\rightarrow \pm \infty} u(x,t) = 0\right\}; & &\|u\|_X = \max_{(x,t)\in \Rm\times [0,T]} |u(x,t)|.&
\end{align*}
Here $T$ is a positive constant to be determined later. We then define a transformation $\mathbf{T}: X \rightarrow X$ by D'Alembert formula
\[
 (\mathbf{T} u)(x,t) = \frac{1}{2}\left[u_0(x-t)+u_0(x+t) +\int_{x-t}^{x+t} u_1(x') dx'\right] - \frac{1}{2} \int_{0}^t \int_{x-t+t'}^{x+t-t'} |u|^{p-1} u(x',t') dx' dt'. 
\]
In addition, if $\|u_1\|_X, \|u_2\|_X \leq A$, then we have 
\begin{align*}
 \|\mathbf{T} u_1\|_X, \|\mathbf{T} u_2\|_X & \leq A/2 + T^2 A^p \\
 \|\mathbf{T} u_1 - \mathbf{T} u_2\| & \leq p T^2 A^{p-1} \|u_1-u_2\|_X. 
\end{align*}
Thus $\mathbf{T}$ is a contraction map from $\{u\in X: \|u\|_X \leq A\}$ to itself if $T$ is sufficiently small. This immediately gives the existence and uniqueness of local solutions. If the local solution $u$ is sufficiently smooth and compactly-supported, a simple calculation immediately gives the energy conservation law. In the general case we need to apply the standard smooth approximation and cut-off techniques. Here we omit the details. 

\paragraph{Global existence} A careful review of the argument above shows that a positive lower bound of the minimal existence time $T$ can be determined solely by the energy $E$. Combining this fact with the energy conservation law, we obtain global solutions $u \in C(\Rm_x \times \Rm_t)$ with $(u(\cdot,t),u_t(\cdot, t)) \in C(\Rm; (\dot{H}^1\cap L^{p+1})\times L^2)$. A Gagliardo-Nirenberg inequality also shows $u(x,t)$ is a bounded function on the whole plane. In addition, if initial data also satisfy $(u_0,u_1) \in C^2(\Rm)\times C^1(\Rm)$, then the solution $u$ is also a $C^2$ function, i.e. a classic solution of the equation.  

\paragraph{Asymptotic behaviour} The asymptotic behaviour of solutions are much difficult to understand. The major reason is that the linear solution in one-dimensional case does not have any decay as $t\rightarrow \infty$. A typical free wave consists of two wave profiles moving to left and right directions at constant speed, i.e. $u(x,t) = f(x-t) + g(x+t)$. As a result, we can no longer solve a nonlinear equation globally by viewing it as a small perturbation of linear equation, even if the initial data are small. This is in contrast to the higher dimensional case. The author would like to mention that decay estimates exist for one-dimensional wave equation $u_{tt} - u_{xx} + V(x) u = 0$ with a suitable potential $V(x)$. Please see D'ancona-Pierfelice \cite{potential1Df}, Donninger-Schlag \cite{potential1D}, for instance. 

\paragraph{Known results} In contrast to the higher dimensional case $d\geq 3$, much less is known about the asymptotic behaviour of solutions to one-dimensional wave equation. H. Lindblad and T. Tao \cite{Ltao1D} proved the following average $L^\infty$ decay 
\[
 \lim_{T\rightarrow +\infty} \frac{1}{T}\int_0^T \|u(\cdot,t)\|_{L^\infty(\Rm)} dt = 0
\]
by applying a quantitative version of Rademacher differentiation theorem. Recently D. Wei and S. Yang \cite{yang1D} observed the following Morawetz-type estimate 
\[
 \int_0^{+\infty} \int_{-t-1}^{t+1} \frac{((t+1)^2-x^2)|u(x,t)|^{p+1}}{(t+1)^3} dx dt < C(E)
\]
and proved that any solution with a finite energy satisfies the following point-wise decay properties
\begin{align}
 &\lim_{t\rightarrow \infty} \|u(\cdot,t)\|_{L^{p+1}(\Rm)} = 0;& &\lim_{t\rightarrow \infty} \|u(\cdot,t)\|_{L^{\infty}(\Rm)} = 0.&
\end{align}
They also gave point-wise polynomial-rate decay of solutions if the initial data decay at a certain rate near the infinity. 

\paragraph{Focusing equations} The author would like to mention that there are many results about the blow-up behaviour of solutions to the focusing equation $u_{tt} - u_{xx} = |u|^{p-1} u$. H. Levine \cite{negativeenergy} proved the finite-time blow-up of solutions as long as the initial data $(u_0,u_1) \in \dot{H}^1 \times L^2$ come with a negative energy
\[
 \int_{-\infty}^\infty \left(\frac{1}{2}|u'_0(x)|^2 + \frac{1}{2}|u_1(x)|^2 - \frac{1}{p+1}|u_0(x)|^{p+1}\right) dx < 0. 
\]
The behaviour of blow-up solutions near blow-up points (blow-up profiles and characteristic points) can be found in Merle-Zaag \cite{blowupprofile, blowupclass, isolatedness} and citations therein. 

\section{Main Results}
The majority of this work is devoted to locating the energy of solutions as $t\rightarrow \infty$. Now we give our main results here. All the details can be found in later sections. By the energy and momentum conservation laws, we may obtain another two conserved quantities: left-going energy $E_-$ and right-going energy $E_+$. 
\begin{align*}
 E_-(u,u_t) & = \int_{-\infty}^\infty \left(\frac{1}{4}|u_x(x,t)+u_t(x,t)|^2 + \frac{1}{2(p+1)}|u(x,t)|^{p+1}\right) dx = \frac{E+M}{2};\\
 E_+(u,u_t) & = \int_{-\infty}^\infty \left(\frac{1}{4}|u_x(x,t)-u_t(x,t)|^2 + \frac{1}{2(p+1)}|u(x,t)|^{p+1}\right) dx = \frac{E-M}{2}.
\end{align*}
We may follow a similar argument to the inward/outward energy theory (dimension $d\geq 3$) introduced in the author's previous works \cite{shenenergy, shen3dnonradial, shenhd}, and obtain the following energy distribution properties of global solutions to the non-linear equation \eqref{cp1}, i.e. almost all energy moves to infinity at a speed close to the light speed. Please note that this property is similar to the higher dimensional case $d\geq 3$. 
\begin{theorem}[Energy distribution] \label{energy distribution}
 Let $u$ be a solution to \eqref{cp1} with a finite energy $E$. Then for any given constant $c\in (0,1)$, we have 
 \begin{align*}
  \lim_{t\rightarrow +\infty} & \int_{-\infty}^{ct} \left(\frac{1}{4}|u_x(x,t)-u_t(x,t)|^2 + \frac{1}{2(p+1)}|u(x,t)|^{p+1}\right) dx = 0;\\
  \lim_{t\rightarrow +\infty} & \int_{-ct}^{\infty} \left(\frac{1}{4}|u_x(x,t)+u_t(x,t)|^2 + \frac{1}{2(p+1)}|u(x,t)|^{p+1}\right) dx = 0.
 \end{align*}
 As a consequence we have 
 \[
  \lim_{t\rightarrow \pm \infty} \int_{-c|t|}^{c|t|} \left(\frac{1}{2}|u_x(x,t)|^2 + \frac{1}{2}|u_t(x,t)|^2 + \frac{1}{p+1}|u(x,t)|^{p+1}\right) dx = 0.
 \]
\end{theorem}
\begin{remark} \label{tail is small}
 By the classic energy flux formula of full energy, we always have 
 \begin{align*}
  \lim_{R\rightarrow +\infty} & \sup_{t\in\Rm} \int_{|x|>|t|+R} \left(\frac{1}{2}|u_x(x,t)|^2 + \frac{1}{2}|u_t(x,t)|^2 + \frac{1}{p+1}|u(x,t)|^{p+1}\right) dx\\
   & \leq  \lim_{R\rightarrow +\infty} \int_{|x|>R} \left(\frac{1}{2}|u'_0(x)|^2 + \frac{1}{2}|u_1(x)|^2 + \frac{1}{p+1}|u_0(x)|^{p+1}\right) dx = 0.
 \end{align*}
 Thus Theorem \ref{energy distribution} implies that eventually all the energy is located near the light cone $|x|=|t|$. 
\end{remark}
\paragraph{Energy retraction} As an application of the energy flux formula given in Section \ref{sec: LR energy theory}, we may also show that any given light cone will eventually contain some amount of energy as $t$ tends to infinity. More precisely we have (The situation of backward light cones is similar.)
\begin{proposition} \label{proposition energy retraction}
 Let $u$ be a nonzero solution to \eqref{cp1} with a finite energy. Then for any given real number $\eta$, we have 
 \[
  \lim_{t\rightarrow +\infty} \int_{|x|<t-\eta} \left(\frac{1}{2}|u_x(x,t)|^2 + \frac{1}{2}|u_t(x,t)|^2 + \frac{1}{p+1}|u(x,t)|^{p+1}\right) dx > 0.
 \]
\end{proposition}
\begin{remark}
 An energy flux formula shows that the integral above is an increasing function of $t$. Thus the conclusion is equivalent to saying that $u$ can not vanish in the light cone $\{(x,t): |x|<t-\eta\}$. This is not trivial. In fact, it is possible to find a nonzero solution to the 3-dimensional wave equation 
\begin{equation} \label{3dwave}
 \partial_t^2 u - \Delta u = - u^5,  \quad (x,t) \in \Rm^3 \times \Rm, 
\end{equation}
so that $u$ vanishes in a given light cone $\{(x,t): |x|<t-\eta\}$. Choose nonzero smooth initial data $(v_0,v_1)$ supported in $\{x: |x|<\eta\}$ and let $v$ be the solution to the linear wave equation with initial data $(v_0,v_1)$. Then by strong Huygens principle, $v$ always vanishes in the light cone $\{(x,t): |x|<t-\eta\}$. Let $u$ be the non-linear profiles associated to $(v,+\infty)$, i.e. a solution to \eqref{3dwave} satisfying
\[
 \lim_{t\rightarrow +\infty} \left\|(u(\cdot,t), u_t(\cdot,t))-(v(\cdot,t), v_t(\cdot,t))\right\|_{\dot{H}^1 \times L^2(\Rm^3)} = 0.
\] 
For existence of non-linear profiles, please refer to Kenig-Merle \cite{kenig1}. By Sobolev embedding, we also have $\|u-v\|_{L^6(\Rm^3)} \rightarrow 0$ as $t\rightarrow +\infty$.  Combining these convergence with the fact that $v$ vanish in the light cone, we have 
\[
 \lim_{t\rightarrow +\infty} \int_{|x|<t-\eta} \left(\frac{1}{2}|u_x(x,t)|^2 + \frac{1}{2}|u_t(x,t)|^2 + \frac{1}{6}|u(x,t)|^{6}\right) dx = 0.
\]
This means that $u$ vanishes in the light cone, since the integral above is again an increasing function of $t$ by energy flux formula. 
\end{remark}
\paragraph {Retraction conjecture}
 We conjecture that all energy will eventually retract in any given light cone. Namely given any $\eta \in \Rm$, we have
 \[
  \lim_{t\rightarrow \infty} \int_{|x|<t-\eta} \left(\frac{1}{2}|u_x(x,t)|^2 + \frac{1}{2}|u_t(x,t)|^2 + \frac{1}{p+1}|u(x,t)|^{p+1}\right) dx = E.
 \]
 This is equivalent to saying 
 \begin{align}
  \lim_{t\rightarrow +\infty} & \int_{t-\eta}^{\infty} \left(\frac{1}{4}|u_x(x,t)-u_t(x,t)|^2 + \frac{1}{2(p+1)}|u(x,t)|^{p+1}\right) dx = 0; \label{conjecture1} \\
  \lim_{t\rightarrow +\infty} & \int_{-\infty}^{\eta-t} \left(\frac{1}{4}|u_x(x,t)+u_t(x,t)|^2 + \frac{1}{2(p+1)}|u(x,t)|^{p+1}\right) dx = 0.
 \end{align}
 In these two limits we only consider right/left-going energies because we have already known that the left/right-going energy gradually vanishes in the right/left half of the space as $t$ tends to infinity. We are not able to prove these limits, but we believe that they are probably true because 
 \begin{itemize}
  \item The integrals above are both decreasing functions of $t$ by trapezoid law Proposition \ref{trapezoid law}; 
  \item The function $u_x(t+s,t)-u_t(t+s,t)$ converges weakly to zero in $L_s^2((-\eta,+\infty))$. We have already known that $\|u\|_{L^{p+1}} \rightarrow 0$. Thus the limits above hold if we might verify the convergence mentioned above is actually a strong limit in the space $L_s^2((-\eta,+\infty))$. 
  \item Although the general case is still unknown, we are at least able to verify that all self-similar solutions $u(x,t) = x^{-2/(p-1)} f(t/x)$ satisfies \eqref{conjecture1}. Here $f \in C^2((-1,+1))$ satisfies a suitable ordinary differently equation. Please note that $u$ can be defined only when $x>t$ because its initial data $(u_0,u_1) = (f(0) x^{-2/(p-1)}, f'(0)x^{-2/(p-1)-1})$ has a singularity at $x=0$. Thus we also assume $\eta<0$. Although these test functions are not globally defined in $\Rm_x \times \Rm_t$, this test is still meaningful by finite speed of propagation.  
 \end{itemize}
 More details can be found in the appendix. 
\begin{remark}
 The retraction conjecture is true for solutions to 1D linear Klein-Gordon equations $u_{tt}-u_{xx}+mu = 0$. Here $m>0$ is a constant. Without loss of generality let us assume $m=1$. By smooth approximation and cut-off techniques we may also assume that initial data $(u_0,u_1)$ are smooth and compact-supported. We combine the following standard dispersive decay estimate for the Klein-Gordon propagator (see, for instance, H\"{o}rmander \cite{hormander})
 \[
  \left\|e^{\pm it\sqrt{-\Delta +1}} f\right\|_{L_x^{\infty}(\Rm)} \lesssim (1+t^2)^{-1/2} \|(-\Delta +1) f\|_{L_x^1}
 \]
 with finite speed of propagation to obtain 
 \[
  \int_{|x|>t-\eta} |u(x,t)|^2 dx \lesssim t^{-1}, \qquad t\gg 1.  
 \]
 Since $u_x, u_t$ are also solutions to 1D linear Klein-Gordon equation with smooth and compact-supported initial data, they satisfy a similar decay estimate. In summary we have
\[
  \int_{|x|>t-\eta} (|u(x,t)|^2 + |u_x(x,t)|^2 + |u_t(x,t)|^2)dx \lesssim t^{-1}, \qquad t\gg 1.  
 \]
\end{remark}
\paragraph{Application on focusing equation} One-dimensional non-linear wave equation has an interesting property: a focusing equation becomes a defocusing equation if we switch $t$ and $x$, and vice versa. This gives us an application of left/right-going energy theory on the focusing wave equation. 
\begin{corollary} \label{corollary focusing}
 Let $u$ be a solution to the focusing one-dimensional wave equation $u_{tt} - u_{xx} = |u|^{p-1} u$ with smooth and compactly-supported initial data. Then $u$ satisfies either of the following
 \begin{itemize}
  \item The solution $u$ blows up in finite time in the positive time direction;
  \item The solution $u$ is defined for all time $t\in [0,\infty)$ and satisfies
  \[
   \lim_{t\rightarrow +\infty} \|(u(\cdot,t), u_t(\cdot,t))\|_{\dot{H}^1 \times L^2} = + \infty. 
  \]
 \end{itemize}
\end{corollary}
\begin{remark}
 Local theory shows that if $u$ blows up in finite time $T_+>0$, then we also have $\|(u(\cdot,t), u_t(\cdot,t))\|_{\dot{H}^1 \times L^2}\rightarrow +\infty$ as $t\rightarrow T_+$. Thus the conclusion of Corollary \ref{corollary focusing} can be rewritten as: the norm  $\|(u(\cdot,t), u_t(\cdot,t))\|_{\dot{H}^1 \times L^2}$ eventually blows up as $T$ approaches the blow-up time $T_+$. ($T_+ = +\infty$ if $u$ is a global solution) This kind of solutions are usually called Type I blow-up solutions. In contrast, a Type II blow-up solution is a global non-scattering solution satisfying 
 \[
  \sup_{t\geq 0} \|(u(\cdot,t), u_t(\cdot,t))\|_{\dot{H}^1 \times L^2} < +\infty.
 \]
 Type II blow-up solutions to focusing, energy critical wave equation in dimension 3 or higher have been intensively studied. See, for instance, Duychaerts-Jia-Kenig \cite{djknonradial} and Duychaerts-Kenig-Merle \cite{dkmradial, se, oddhigh}. 
\end{remark}
\paragraph{Energy concentration} A natural question about energy distribution is whether the energy may concentrates in a small region. The following result is a first attempt to investigate this problem. It shows that under suitable conditions you can not find a series of intervals $(x(t), x(t)+\delta(t))$ depending on time so that (i) the length $\delta(t)$ of intervals satisfies $\delta(t) \rightarrow 0$ as $t\rightarrow \infty$; (ii) The right-going energy outside the interval $(x(t), x(t)+\delta(t))$ decays fast. (i.e. $o(1/t)$)
\begin{proposition} \label{prop interaction}
 Let $u$ be a solution to \eqref{cp1} with compactly-supported and even\footnote{Even data in the one-dimensional case correspond to radial data in higher dimensions} initial data. Then we have 
\[
 \liminf_{t\rightarrow +\infty} \iint_{\Rm^2} |x_1-x_2| e_+(x_1,t) e_+(x_2,t) dx_1 dx_2 > 0. 
\]
Here $e_+(x,t)$ is the right-going energy density given in Definition \ref{def rl energy}. 
\end{proposition}

\section{Left/Right-going Energy Theory} \label{sec: LR energy theory}
We start by introducing some notations for our convenience. 
\begin{definition}[Right/left-going energy] \label{def rl energy}
 We first define the right/left-going energy density. 
 \begin{align*}
  &e_+(x,t) = \frac{1}{4}|u_x(x,t)-u_t(x,t)|^2 + \frac{1}{2(p+1)}|u(x,t)|^{p+1}& \\
   &e_-(x,t) = \frac{1}{4}|u_x(x,t)+u_t(x,t)|^2 + \frac{1}{2(p+1)}|u(x,t)|^{p+1}&
 \end{align*}
 Thus the integral $\int_\Rm e_\pm (x,t) dx$ defines the right/left-going energy and is a constant for all $t\in \Rm$. We use the following notation to represent the right/left-going energy in a given region $J \subset \Rm$ at time $t$
 \[
  E_\pm (t; J) = \int_J e_\pm (x,t) dx. 
 \]
 In particular, if $J=(a,b)$ is an interval ($-\infty \leq a \leq b\leq \infty$), we use the notation $E_\pm (t; a,b) = E_\pm(t; J)$.  
\end{definition}
\noindent The main tool of our left/right-going energy theory is the following energy flux formula
\begin{proposition}[energy flux] \label{energy flux formula}
 Let $\Gamma$ be a simple, closed curve in $\Rm_x \times \Rm_t$. For convenience we also assume that $\Gamma$ consists of finite line segments, each of which is paralleled to $x$-axis, $t$-axis or light rays $t\pm x=0$. Then we have 
 \begin{align}
  &\int_\Gamma e_+(x,t) dx + \left[-\frac{1}{4}|u_x(x,t)-u_t(x,t)|^2 + \frac{1}{2(p+1)}|u(x,t)|^{p+1}\right] dt = 0; \label{flux right} \\
  &\int_\Gamma e_-(x,t) dx + \left[+\frac{1}{4}|u_x(x,t)+u_t(x,t)|^2 - \frac{1}{2(p+1)}|u(x,t)|^{p+1}\right] dt = 0. \nonumber
 \end{align}
\end{proposition}
\begin{proof}
 If the solution is $C^2$, then these identities immediately follow an application of Green's formula. Because a simple calculation shows
\begin{align*}
 \frac{\partial}{\partial t} e_+(x,t) &= \frac{\partial}{\partial x} \left[-\frac{1}{4}|u_x(x,t)-u_t(x,t)|^2 + \frac{1}{2(p+1)}|u(x,t)|^{p+1}\right]; \\
 \frac{\partial}{\partial t} e_-(x,t) &= \frac{\partial}{\partial x} \left[+\frac{1}{4}|u_x(x,t)+u_t(x,t)|^2 - \frac{1}{2(p+1)}|u(x,t)|^{p+1}\right]. 
\end{align*}
If the solution is not sufficiently smooth, we also need to apply standard smooth approximation techniques. Here we skip the details. 
\end{proof}
\begin{remark}
 We may also consider more general simple, closed curves. But the choices here are sufficient for our application. All the line integrals involved are well-defined and finite because 
 \begin{itemize}
  \item The solution $u$ is continus and bounded in $\Rm_x \times \Rm_t$;
  \item The identity $(\partial_x + \partial_t)(u_x(x,t)-u_t(x,t)) = -|u|^{p-1}u$ shows that $u_x(x,t)-u_t(x,t)$ is Lipschitz (with a universal Lipschitz constant) along the lines $x-t=\eta, \, \eta \in \Rm$ except for $\eta$ in a set of measure zero. For those $\eta$ singularity may be introduced by initial data and carried on along light rays. Please refer to Reed \cite{singularity} for more details on the propagation of singularities along light rays. In addition, $|u_x(x,t)-u_t(x,t)|^2$ is integrable in any compact subset of $\Rm_x \times \Rm_t$. Thus the line integral of $|u_x(x,t)-u_t(x,t)|^2$ along a line segment is always meaningful and finite as long as the line segment is not paralleled to $x-t = 0$. The situation of $|u_x(x,t)+u_t(x,t)|^2$ is similar. Another way to show the integrability of $|u_x(x,t)\pm u_t(x,t)|^2$ along the line segments is to observe that the contribution of nonlinear term in D'Alembert formula
 \[
  -\frac{1}{2} \int_0^t \int_{x-t+t'}^{x+t-t'} |u|^{p-1} u(x',t') dx' dt'
 \]
 is always a $C^1$ function. Thus it suffices to consider the linear propagation of initial data. 
 \end{itemize}
\end{remark}

 \begin{figure}[h]
 \centering
 \includegraphics[scale=0.75]{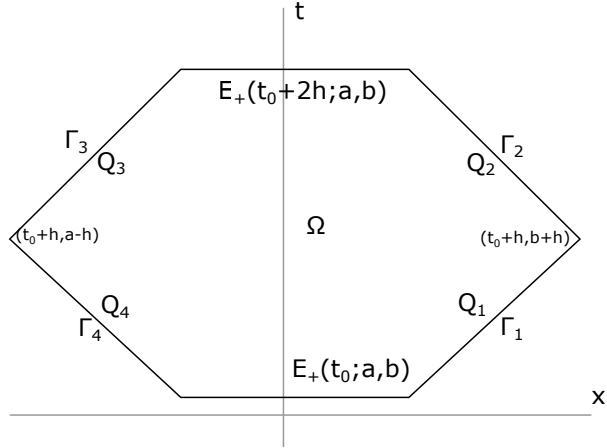}
 \caption{Illustration of integral path} \label{figure flux1D}
\end{figure}

\paragraph{Physical interpretation} Now we explain why the identities given above are actually energy flux formula and explain the physical meanings of line integral along different types of line segments. We do this by considering an example. Let $\Gamma$ be the boundary of the region shown in figure 
\ref{figure flux1D}. Then we may write the energy flux formula for right-going energy \eqref{flux right} in details:
\begin{align*}
 E_+(t+2h; a, b) & = E_+(t; a,b) + Q_1 - Q_2 - Q_3 + Q_4.
\end{align*}
\begin{align*}
 &Q_1 = \frac{1}{p+1} \int_{t_0}^{t_0+h} |u(b\!-\!t_0\!+\!t,t)|^{p+1} dt;& &Q_2 = \frac{1}{2} \int_{t_0+h}^{t_0+2h} |(u_x-u_t)(b\!+\!t_0\!+\!2h\!-\!t,t)|^2 dt;& \\
 &Q_3 = \frac{1}{p+1} \int_{t_0+h}^{t_0+2h} |u(a\!-\!t_0\!-\!2h\!+\!t,t)|^{p+1} dt;& &Q_4 = \frac{1}{2} \int_{t_0}^{t_0+h} |(u_x-u_t)(a\!+\!t_0\!-\!t,t)|^2 dt.&
\end{align*}
\noindent The difference of right-going energies is a sum of four terms. $Q_1$ is the amount of energy gained along the line segment $\Gamma_1$ due to non-linear effect; $Q_2$ is the amount of energy moving outside the region across $\Gamma_2$ by the linear propagation; $Q_3$ is the amount of energy which ``leaks'' along the line segment $\Gamma_3$ due to non-linear effect; Finally $Q_4$ is the amount of energy moving inside the region across $\Gamma_4$ by the linear propagation. 

\paragraph{Trapezoid law} The following proposition will be frequently used in the later part of this work:
\begin{proposition}[Trapezoid law] \label{trapezoid law}
 Let $\eta, t_1, t_2 \in \Rm$. We have 
 \begin{align}
  E_+(t_2; -\infty, t_2-\eta) - E_+(t_1; -\infty, t_1-\eta) & = E_+(t_1; t_1-\eta, \infty) - E_+(t_2; t_2-\eta, \infty) \nonumber \\
   & = \frac{1}{p+1} \int_{t_1}^{t_2} |u(t-\eta,t)|^{p+1} dt; \label{trapezoid1} \\
  E_-(t_2; -\infty, t_2-\eta) - E_-(t_1; -\infty, t_1-\eta) & = E_-(t_1; t_1-\eta, \infty) - E_-(t_2; t_2-\eta, \infty) \nonumber \\
   & = \frac{1}{2} \int_{t_1}^{t_2} |u_x(t-\eta,t)+u_t(t-\eta,t)|^2 dt. \nonumber
 \end{align}
\end{proposition}

 \begin{figure}[h]
 \centering
 \includegraphics[scale=1.2]{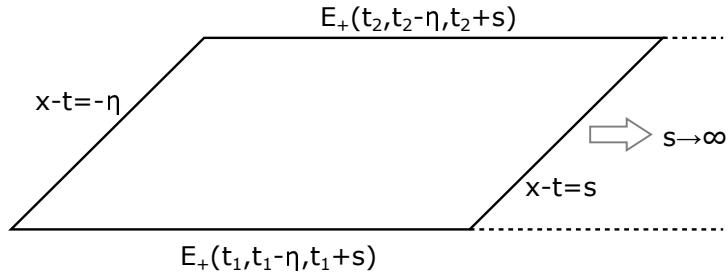}
 \caption{Illustration of parallelogram} \label{figure parallel2}
\end{figure}

\begin{proof}
 Let $\Gamma$ be the boundary of the parallelogram $\Omega = \{(x,t): -\eta<x-t<s, t_1<t<t_2\}$, as shown in figure \ref{figure parallel2}. Here $s$ is a large parameter. We apply the energy flux formula of right-going energy on $\Gamma$ and obtain 
\begin{align}
 E_+(t_1; t_1-\eta, t_1+s) + \frac{1}{p+1} \int_{t_1}^{t_2} |u(t+s,t)|^{p+1} dt  & - E_+(t_2; t_2-\eta, t_2+s)\nonumber\\
  & -  \frac{1}{p+1} \int_{t_1}^{t_2} |u(t-\eta,t)|^{p+1} dt = 0. \label{parallelogram1}
\end{align}
We observe the fact 
\[
 \int_{-\infty}^\infty \left(\int_{t_1}^{t_2} |u(t+s,t)|^{p+1} dt \right)ds = \iint_{\Rm_x \times [t_1,t_2]} |u(x,t)|^{p+1} dx dt \leq (t_2-t_1)(p+1)E < +\infty. 
\]
Thus we have the lower limit
\[
 \liminf_{s \rightarrow +\infty} \int_{t_1}^{t_2} |u(t+s,t)|^{p+1} dt = 0.
\]
This enables us to make $s\rightarrow +\infty$ in \eqref{parallelogram1} and obtain the second half of \eqref{trapezoid1}. The first half of this identity immediately follows the conservation law: 
\[
 E_+(t_1; -\infty, t_1-\eta) + E_+(t_1; t_1-\eta, \infty)  = E_+ = E_+(t_2; -\infty, t_2 - \eta) + E_+(t_2; t_2-\eta, \infty).
\]
Thus we finish the proof of \eqref{trapezoid1}. The second identity can be proved in the same way. 
\end{proof}
\section{Energy Distribution}
In this section we prove Theorem \ref{energy distribution}. It suffices to consider the asymptotic behaviour of right-going energy as $t\rightarrow +\infty$. Other cases can be handled with by symmetry. More precisely, we need to prove the following limit for any given real number $c\in (0,1)$. 
\[
 \lim_{t\rightarrow +\infty} E_+(t; -\infty, ct) = 0.
\]
Fix a large time $t$ and let $d= (1-c)t/2$. Then for any $s \in [0,d]$ we may apply trapezoid law and obtain
\begin{align*}
 E_+(t; -\infty, ct) & \leq E_+(t, -\infty, ct+s) \\
 & = E_+(0;-\infty, -2d+s) + \frac{1}{p+1} \int_0^{t} |u(-2d+s+t', t')|^{p+1} dt'\\
 & \leq E_+(0;-\infty, -d)  + \frac{1}{p+1} \int_0^{t} |u(-2d+s+t', t')|^{p+1} dt'.
\end{align*}
Integrating this inequality for $s \in [0,d]$ we obtain
\begin{align*}
 d E_+(t; -\infty, ct) & \leq d E_+(0; -\infty, -d) + \frac{1}{p+1} \int_0^d \int_0^{t} |u(-2d+s+t', t')|^{p+1} dt' ds\\
 & = d E_+(0; -\infty, -d) + \frac{1}{p+1} \iint_{\Omega} |u(x, t')| dx dt' \\
 & \leq d E_+(0; -\infty, -d) + \frac{1}{p+1} \int_{0}^{t} \int_{-\infty}^\infty |u(x,t')|^{p+1} dx dt.
\end{align*}
Here the region $\Omega$ is a parallelogram with vertices $(-d,0)$, $(-2d,0)$, $(ct, t)$ and $(ct+d,t)$, as shown in figure \ref{figure energydis2}. Recalling $d = (1-c)t/2$ and dividing both sides of the inequality above by $d$, we have 
\[
 E_+(t; -\infty, ct) \leq E_+(0; -\infty, -\tfrac{1-c}{2}t) + \frac{2}{(p+1)(1-c)t} \int_0^{t} \int_{-\infty}^\infty |u(x,t')|^{p+1} dx dt'
\]
Finally we make $t\rightarrow +\infty$ and utilize the already known fact (see \cite{yang1D})
\[
 \lim_{t\rightarrow \infty} \|u(\cdot,t)\|_{L^{p+1}(\Rm)} = 0
\]
to finish the proof.

 \begin{figure}[h]
 \centering
 \includegraphics[scale=0.9]{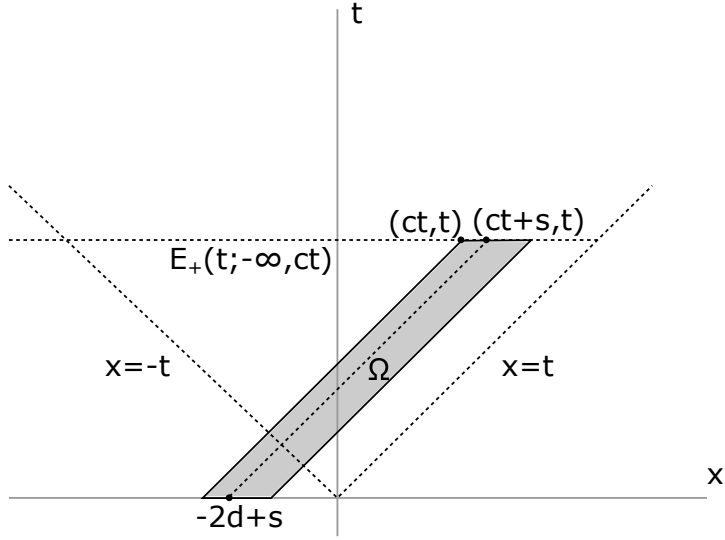}
 \caption{Illustration of proof} \label{figure energydis2}
\end{figure}

\section{Energy Retraction to Light Cones}
In this section we prove Proposition \ref{proposition energy retraction} and Corollary \ref{corollary focusing}. We start by introducing a lemma. 
\begin{lemma} \label{vanishing lemma}
 Let $u$ be a solution to \eqref{cp1} with a finite energy and $(x_0,t_0)$ be a point in $\Rm_x \times \Rm_t$. If $u(x_0+s, t_0+s) = 0$ for all $s \geq 0$, i.e. $u$ vanishes on the light ray starting at $(x_0,t_0)$, then we have $u(x,t) = 0$ for all $(x,t)$ with $x\geq |t-t_0|+x_0$. 
\end{lemma}
\begin{proof}
 For any $s_0>0$ we may apply trapezoid law and obtain 
 \[
  E_-(t_0; x_0, +\infty) - E_-(t_0+s_0; x_0+s_0, +\infty) = \frac{1}{2} \int_{0}^{s_0} |(u_x+u_t)(x_0+s, t_0+s)|^2 ds = 0.
 \]
 By Theorem \ref{energy distribution} we have 
 \[
  \lim_{s_0\rightarrow +\infty} E_-(t_0+s_0; x_0+s_0, +\infty) = 0. 
 \]
 Thus we obtain $E_-(t_0; x_0, +\infty) = 0$. This implies that $(u(t_0,x), u_t(t_0,x)) = (0,0)$ for all\footnote{$u_t=0$ in the sense of $L^2$ functions, or almost everywhere $x>x_0$.} $x > x_0$. The conclusion immediately follows because of the finite speed of propagation. 
\end{proof}
\begin{remark} \label{another ray}
 In the same manner we may prove the following result: if $u$ is a solution to \eqref{cp1} with a finite energy so that $u(x_0-s, t_0+s) = 0$ for all $s \geq 0$, then $u(x,t) = 0$ for all $(x,t)$ with $x\leq -|t-t_0|+x_0$.
\end{remark}
\paragraph{Proof of Proposition \ref{proposition energy retraction}} According to energy flux formula, the energy contained in the light cone 
 \[
  E_\eta(t) \doteq \int_{|x|<t-\eta} \left(\frac{1}{2}|u_x(x,t)|^2 + \frac{1}{2}|u_t(x,t)|^2 + \frac{1}{p+1}|u(x,t)|^{p+1}\right) dx
 \]
is an increasing function of $t$. Thus it suffices to show $E_\eta(t) > 0$ for at least one $t>\eta$. If this were false, we would have 
\begin{equation} \label{vanish in the cone}
 u(x,t) = 0, \quad \hbox{if}\;  |x|\leq t-\eta. 
\end{equation}
The identity above holds on the boundary $|x|=t-\eta$ as well because $u$ is a continuous function of $(x,t)\in \Rm^2$. Since $u$ vanishes on the light ray $x= t-\eta, t\geq \eta$, we may apply Lemma \ref{vanishing lemma} and conclude that 
\begin{equation} \label{vanish right}
 u(x,t) = 0, \quad \hbox{if}\; x\geq |t-\eta|.
\end{equation}
We also have $u$ vanishes on the light ray $x=\eta-t, t\geq \eta$. By Remark \ref{another ray} we obtain  
\begin{equation} \label{vanish left}
 u(x,t) = 0, \quad \hbox{if}\; x\leq -|t-\eta|.
\end{equation}
Combining \eqref{vanish in the cone}, \eqref{vanish right} and \eqref{vanish left},  we obtain $u(x,t)=0$ for all $x\in \Rm$ and $t \geq \eta$. This contradicts with our assumption that $u$ is a nonzero solution. 

\paragraph{Application on focusing equation} At the end of this section we prove Corollary \ref{corollary focusing}. Assume $u(y,s)$ solves $u_{ss} - u_{yy} = |u|^{p-1} u$. If neither (i) nor (ii) held, then the solution $u(y,s)$ would be defined for all $s>0$ so that 
\begin{equation}
 \liminf_{s \rightarrow +\infty} \|(u(\cdot,s), u_s(\cdot,s))\|_{\dot{H}^1\times L^2} < +\infty.  \label{bounded H1L2}
\end{equation}
Since the initial data are compactly-supported, we may find a real number $R>0$, so that $\hbox{Supp} (u_0,u_1)\subset (-R,R)$. By finite speed of propagation, we have
\begin{equation}
 u(y,s) = 0, \quad \hbox{if}\; |y|\geq s + R, \, s\geq 0. \label{finite support}
\end{equation}
In addition, because the initial data are assumed to be smooth, the solution $u$ is at least a $C^2$ solution defined in a neighbourhood of $\Rm_y \times [0,+\infty)$. The idea is to consider the $C^2$ solution $v(x,t) = u(t,x)$ to the defocusing equation defined in a neighbourhood of $[0,+\infty) \times \Rm_t$ and to apply the left/right-going energy theory. We first need to determine whether $v$ still comes with a finite energy. Let us fix a function
\[
 a(x) = \left\{\begin{array}{ll} y, & |y| \leq R; \\ R, & |y| \geq R; \end{array}\right.
\]
and define 
\[
 I(s) = \int_{-\infty}^\infty a(y) u_y(y,s) u_s(y,s) dy.
\]
A straight-forward calculation shows that
\begin{align*}
 I'(s) & = \int_{-\infty}^\infty [a(y) u_{ys}(y,s) u_s(y,s) + a(y) u_y(y,s) u_{ss}(y,s)] dy\\
 & = \int_{-\infty}^\infty [a(y) u_{ys}(y,s) u_s(y,s) + a(y) u_y(y,s) u_{yy}(y,s) + a(y) u_y(y,s) |u|^{p-1} u(y,s)] dy\\
 & = \int_{-\infty}^\infty a(y) \partial_y \left[\frac{1}{2} u_s^2(y,s) + \frac{1}{2} u_y^2(y,s) + \frac{1}{p+1} |u(y,s)|^{p+1} \right] dy\\
 & = - \int_{-R}^R \left[\frac{1}{2} u_s^2(y,s) + \frac{1}{2} u_y^2(y,s) + \frac{1}{p+1} |u(y,s)|^{p+1} \right] dy.
\end{align*} 
Observing the fact $|I(s)| \leq (R/2) \|(u(\cdot,s), u_s(\cdot,s))\|_{\dot{H}^1\times L^2}^2$, we may integrate $I'(s)$ from $0$ to $S$ and obtain an inequality 
\begin{align*}
 \int_0^S \int_{-R}^R & \left[\frac{1}{2} u_s^2(y,s) + \frac{1}{2}u_y^2(y,s) + \frac{1}{p+1} |u(y,s)|^{p+1} \right] dy ds\\
 & \qquad \qquad \leq \frac{R}{2}\left(\|(u_0, u_1)\|_{\dot{H}^1\times L^2}^2 + \|(u(\cdot,S), u_s(\cdot,S))\|_{\dot{H}^1\times L^2}^2\right).
\end{align*}
We recall \eqref{bounded H1L2}, make $S\rightarrow +\infty$ and obtain
\[
 \int_0^{+\infty} \int_{-R}^R  \left[\frac{1}{2} u_s^2(y,s) + \frac{1}{2}u_y^2(y,s) + \frac{1}{p+1} |u(y,s)|^{p+1} \right] dy ds < + \infty. 
\]
This implies that the solution $v(x,t) = u(t,x)$ to the defocusing equation satisfies 
\begin{equation} \label{bounded energy v}
  \int_{-R}^R \int_0^{+\infty} \left[\frac{1}{2} v_x^2(x,t) + \frac{1}{2}v_t^2(x,t) + \frac{1}{p+1} |v(x,t)|^{p+1} \right] dx dt < + \infty. 
\end{equation}
Thus there exists a time $t_0 \in (-R,R)$ so that $E(t_0; 0, +\infty) < +\infty$. Here the energy in the right half line $E(t; 0, +\infty)$ is defined by
\[
 E(t; 0, +\infty) = \int_0^{+\infty} \left[\frac{1}{2} v_x^2(x,t) + \frac{1}{2}v_t^2(x,t) + \frac{1}{p+1} |v(x,t)|^{p+1} \right] dx.
\]
Next we verify $E(t; 0, +\infty) < +\infty$ for all $t\in [-R,R]$. It suffices to prove this inequality for all $t_1\in (t_0,R]$ by symmetry. We apply energy flux formula of full energy\footnote{This can be obtained by combining left/right-going energies together. } on the rectangle $[0,r]\times [t_0,t_1]$ and write
\begin{align*}
 E(t_0; 0,r) + \int_{t_0}^{t_1} v_x(r,t) v_t(r,t) dt - E(t_1; 0, r) - \int_{t_0}^{t_1} v_x(0,t) v_t(0,t) dt = 0. 
\end{align*} 
Thanks to \eqref{bounded energy v}, we may let $r\rightarrow +\infty$ in the identity above and conclude 
\[
 E(t_1; 0, +\infty) = E(t_0; 0, +\infty) - \int_{t_0}^{t_1} u_0' (y) u_1 (y) dy < + \infty.
\]
We may also rewrite \eqref{finite support} in term of $v$: 
\begin{equation} \label{vanishing 1}
 v(x,t) = 0, \quad \hbox{if}\; 0\leq x \leq |t|-R, \; |t| \geq R.
\end{equation}
We next combine this fact with $E(R; 0, +\infty) < +\infty$, apply Lemma \ref{vanishing lemma} and conclude
\begin{equation} \label{vanishing 2}
 v(x,t) = 0, \quad \hbox{if}\; x\geq |t-R|.
\end{equation}
Please note that although the solution $v$ is not necessarily defined for $x<0$, we are still able to apply Lemma \ref{vanishing lemma} here. Since we may consider the solution $\tilde{v}$ to $v_{tt} - v_{xx} = -|v|^{p-1}v$ with initial data $(\tilde{v}_0, \tilde{v}_1)\in (\dot{H}^1 \cap L^{p+1})\times L^2(\Rm)$ at time $t=R$, so that $(\tilde{v}_0(x), \tilde{v}_1(x)) = (v(x,R), v_t(x,R))$ for all $x \geq 0$. By finite speed of propagation we have $v(x,t) = \tilde{v}(x,t)$ if $x\geq |t-R|$. Thus identity $\tilde{v}(x,x+R)=0$ holds for all $x\geq 0$. This enables us to apply Lemma \ref{vanishing lemma} on $\tilde{v}$ and conclude $\tilde{v}(x,t) = 0$ if $x \geq |t-R|$. This immediately verifies \eqref{vanishing 2} because $v$ and $\tilde{v}$ coincide in this region. In the same manner we also have 
\begin{equation} \label{vanishing 3}
 v(x,t) = 0, \quad \hbox{if}\; x\geq |t+R|.
\end{equation}
Combining \eqref{vanishing 1}, \eqref{vanishing 2}, \eqref{vanishing 3} we obtain $v(x,t) = 0$ for any $x\geq R, t\in \Rm$. This means $u(y,s)=0$ for all $y\in \Rm$, $s\geq R$ and gives a contradiction. For readers' convenience, the regions involved in this proof are illustrated in figure \ref{figure switch}. The solution $v(x,t)$ vanishes in the lighter grey regions by finite speed of propagation, in other grey regions by Lemma \ref{vanishing lemma}. 

 \begin{figure}[h]
 \centering
 \includegraphics[scale=0.9]{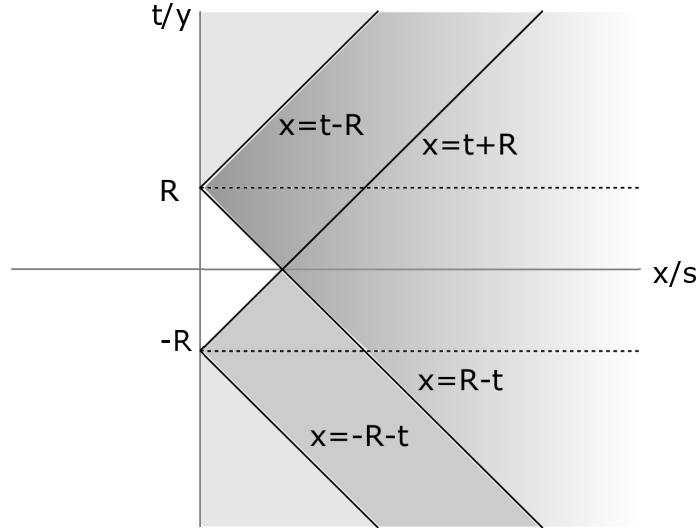}
 \caption{Illustration of regions} \label{figure switch}
\end{figure}

\section{Energy Concentration}  

In this section we prove Proposition \ref{prop interaction}. Let us define 
\[
 Q(t) = \iint_{\Rm^2} |x_1-x_2| e_+(x_1,t) e_+(x_2,t) dx_1 dx_2. 
\]
For convenience we also use the notation 
\[
 e'_+(x,t) = -\frac{1}{4}|u_x(x,t)-u_t(x,t)|^2 + \frac{1}{2(p+1)}|u(x,t)|^{p+1}.
\]
Keeping $\partial_t e_+(x,t) = \partial_x e'_+(x,t)$ in mind, we may calculate 
\begin{align*}
 Q'(t) & = \iint_{\Rm^2} |x_1-x_2| \left[\partial_{x_1} e'_+(x_1,t) e_+(x_2,t) + e_+(x_1,t) \partial_{x_2} e'_+(x_2,t)\right] dx_1 dx_2 \\
 & = \iint_{\Rm^2} \frac{x_1-x_2}{|x_1-x_2|} \left[e_+(x_1,t) e'_+(x_2,t) - e'_+(x_1,t) e_+(x_2,t)\right] dx_1 dx_2\\
 & = \frac{1}{2(p+1)} \iint_{\Rm^2} \frac{x_1-x_2}{|x_1-x_2|} |u_x(x_1,t)-u_t(x_1,t)|^2 |u(x_2,t)|^{p+1} dx_1 dx_2
\end{align*}
We have assumed that the initial data are even, thus $u(x,t) = u(-x,t)$ holds for all $(x,t) \in \Rm^2$. Therefore if $x_1>0$, we have 
\begin{align*}
 \int_\Rm  \frac{x_1-x_2}{|x_1-x_2|} |u(x_2,t)|^{p+1} dx_2 & = \int_{-\infty}^{x_1} |u(x_2)|^{p+1} dx_2 - \int_{x_1}^\infty |u(x_2)|^{p+1} dx_2 \\
 & = \int_{-x_1}^{x_1} |u(x_2)|^{p+1} dx_2 \geq 0.
\end{align*}
Similarly if $x_1<0$, we have 
\[
 \int_\Rm  \frac{x_1-x_2}{|x_1-x_2|} |u(x_2,t)|^{p+1} dx_2 = - \int_{x_1}^{-x_1} |u(x_2)|^{p+1} dx_2 \leq 0.
\]
By Theorem \ref{energy distribution}, there exists a time $T>1$, so that when $t\geq T$, we always have 
\begin{equation} \label{upper bound e1}
 \int_{t/2}^\infty e_+(x_2,t) dx_2 > E_+/2 \quad \Rightarrow \quad \int_{-\infty}^0 e_+(x_1,t) dx_1 \leq \frac{4Q(t)}{E_+ t} 
\end{equation}
By Symmetry we also have 
\begin{equation} \label{upper bound Lp1}
 \int_\Rm |u(x_2)|^{p+1} dx_2 = 2 \int_{-\infty}^0 |u(x_2)|^{p+1} dx_2  \leq 4(p+1) \int_{-\infty}^0 e_+(x_2,t) dx_2 \leq \frac{16(p+1) Q(t)}{E_+ t}. 
\end{equation}
Thus when $t \geq T$, $Q(t)$ satisfies an inequality 
\begin{align*}
 Q'(t) & =  \frac{1}{2(p+1)} \int_{\Rm} \left[|u_x(x_1,t)-u_t(x_1,t)|^2 \int_{\Rm} \frac{x_1-x_2}{|x_1-x_2|} |u(x_2,t)|^{p+1} dx_2 \right] dx_1\\
&  \geq -\frac{1}{2(p+1)} \int_{-\infty}^0 \left[|u_x(x_1,t)-u_t(x_1,t)|^2  \int_{x_1}^{-x_1} |u(x_2,t)|^{p+1} dx_2\right] dx_1\\
& \geq -\frac{8 Q(t)}{E_+ t} \int_{-\infty}^0 |u_x(x_1,t)-u_t(x_1,t)|^2 dx_1\\
& \geq -\frac{128 Q(t)^2}{E_+^2 t^2}. 
\end{align*}
Here in the final step we use inequality \eqref{upper bound e1}. This immediately gives ($t\geq T$)
\[
 [1/Q(t)]' \leq \frac{128}{E_+^2 t^2} \quad \Rightarrow \quad \frac{1}{Q(t)} \leq \frac{1}{Q(T)} + \frac{128}{E_+^2 T} < + \infty, 
\]
thus finishes the proof.

\section{Appendix}
In this final section we give some hints about the retraction conjecture. We first prove 
\begin{lemma}
 Let $u$ be a solution to \eqref{cp1} with a finite energy. Given any $\eta \in \Rm$, the function $u_x(t+s,t)-u_t(t+s,t)$ converges weakly to zero in $L_s^2((-\eta,+\infty))$. 
\end{lemma}
\begin{proof}
 First of all, Theorem \ref{energy distribution} implies that $E_-(t; t-\eta, +\infty) \rightarrow 0$ as $t\rightarrow +\infty$. Thus $u_x(t+s,t)+u_t(t+s,t)$ converges strongly to zero in $L_s^2((-\eta,+\infty))$. As a consequence it suffices to show $u_x(t+s,t)$ converges weakly to zero in $L_s^2((-\eta,+\infty))$. By energy conservation law, the norms $\|u_x(t+s,t)\|_{L_s^2}((-\eta,+\infty))$ are uniformly bounded. Thus we only need to show that for any smooth and compactly-supported function $f(s) \in C_0^\infty((-\eta, +\infty))$, we have
 \[
  \lim_{t\rightarrow +\infty} \int_{-\eta}^\infty u_x(t+s,t) f(s) ds =0.
 \]
 This immediately follows an integration by parts 
 \[
  \int_{-\eta}^\infty u_x(t+s,t) f(s) ds = - \int_{-\eta}^\infty u(t+s,t) f'(s) ds \rightarrow 0.
 \]
 In the final step we use the fact $\|u(t+\cdot,t)\|_{L^{p+1}((-\eta, +\infty))}\leq \|u(\cdot,t)\|_{L^{p+1}(\Rm)} \rightarrow 0$. 
\end{proof}

\paragraph{Self-similar solutions} Next we consider a family of self-similar solutions in the form of $u(x,t) = x^{-2/(p-1)} f(t/x)$. Here $x>t\geq 0$. For convenience we use the notation $\beta = 2/(p-1)$. A simple calculation shows 
\begin{align*}
 u_x (x,t) & = -\beta x^{-\beta-1} f(t/x) - t x^{-\beta -2} f'(t/x); \\
 u_{xx} (x,t) & = \beta (\beta+1) x^{-\beta-2} f(t/x) + 2(\beta+1) t x^{-\beta -3} f'(t/x) + t^2 x^{-\beta -4} f''(t/x); \\
 u_{tt} (x,t) & = x^{-\beta-2} f''(t/x).
\end{align*}
Thus $u$ solves \eqref{cp1} if and only if $f(y)$ solves the ordinary differential equation
\begin{equation} \label{ode}
 (1-y^2) f''(y) - 2(\beta+1) y f'(y) - \beta(\beta+1) f(y) + |f(y)|^{p-1} f(y) = 0.
\end{equation}
The following lemma gives some properties of solutions to this ordinary differential equation. The proof of this lemma is postponed to the final part of this section. 
\begin{lemma} \label{lemma ode}
 The ordinary differential equation
\[
 \left\{\begin{array}{ll} (1-y^2) f''(y) - 2(\beta+1) y f'(y) - \beta(\beta+1) f(y) + |f(y)|^{p-1} f(y) = 0, & y \in (-1,1);\\
 f(0) = a, \; f'(0) =b & \end{array}\right.
\]
has a unique solution $f \in C^2((-1,1))$ for any initial data $(a,b)\in \Rm^2$. In addition, we always have
\[
 \lim_{y\rightarrow 1^-} (1-y)^{1+\beta} f'(y) = 0.
\]
\end{lemma}
\paragraph{Asymptotic behaviour} The lemma gives the existence of self-similar solutions $u$ for any initial data $(u_0,u_1) = (a x^{-\beta}, b x^{-\beta-1})$. We show that given any $R>0$, the right-going energy of these solutions on the right hand of light ray $x = t+ R$ converges to zero when $t\rightarrow +\infty$. Namely 
\[
 \lim_{t\rightarrow +\infty} E_+(t; t+R, +\infty) = 0. 
\]
By the already known fact $E_-(t; t+R, +\infty) \rightarrow 0$ and Remark \ref{tail is small}, it suffices to show 
\[
 \lim_{t\rightarrow +\infty} \int_{t+R}^{t+R_1} |u_t(x,t)|^2 dx =0, \quad \forall R_1>R>0.
\]
We rewrite the integral above in term of $f$ and utilize the asymptotic behaviour of $f'(y)$:
\begin{align*}
 \int_{t+R}^{t+R_1} |u_t(x,t)|^2 dx & = \int_{t+R}^{t+R_1} \left|x^{-\beta-1} f'(t/x)\right|^2 dx = \int_{t+R}^{t+R_1} x^{-2\beta-2} \left(1-\frac{t}{x}\right)^{-2\beta-2} o(1) dx\\
 & \leq \int_{t+R}^{t+R_1} (x-t)^{-2\beta-2} o(1) dx \leq (R_1-R) R^{-2\beta-2} o(1) \rightarrow 0.
\end{align*}
\paragraph{Proof of lemma} Finally we sketch a proof of Lemma \ref{lemma ode}. Because a lot of theories and techniques below are standard in the study of ordinary differential equations, we will skip some details. By standard ODE theory, the solution $f$ must be unique, $C^2$ and defined in a maximal interval $(-\delta_1, \delta_2)$. In addition, if $\delta_i<1$, then either $f(y)$ or $f'(y)$ would blow up as $y\rightarrow \delta_i$. This blow-up can never happen because of a semi-conservation law. Let us define 
\begin{align*}
 P(z) & = C_p - \frac{\beta(\beta+1)}{2} z^2 + \frac{1}{p+1} |z|^{p+1}; \\
 \tilde{E} (y) & = \frac{1}{2} (1-y^2)^{2\beta+2} |f'(y)|^2 + (1-y^2)^{2\beta+1} P(f(y)).
\end{align*}
Here $C_p$ is a positive constant determined by $p$ so that we always have
\begin{equation} \label{lower bound of P}
 P(z) \geq \frac{1}{p+2}|z|^{p+1} \geq 0
\end{equation}
A simple calculation shows that 
\begin{equation} \label{energy decay rate}
 \tilde{E}'(y) = -2(2\beta+1)y(1-y^2)^{2\beta} P(f(y)).
\end{equation}
Thus we have $\tilde{E}(y) \leq \tilde{E}(0) < +\infty$ for all $y$ in the maximal lifespan. This gives the boundedness of $f(y)$ and $f'(y)$ in any closed interval $I \subset (-1,1)$, thus guarantees the existence of solutions in the whole interval $(-1,1)$. This semi-conservation law also implies 
\begin{align}
 &|f'(y)| \lesssim (1-y^2)^{-\beta-1};& &|f(y)| \lesssim (1-y^2)^{-\frac{2\beta+1}{p+1}}.& \label{boundedness of f}
\end{align}
Now let us consider the asymptotic behaviour of $f$ as $y\rightarrow 1^-$. Because $\tilde{E}'(y)$ is a nonnegative, decreasing function of $y\in [0,1)$, the limit \[
 A \doteq \lim_{y \rightarrow 1^-} \tilde{E}(y) \geq 0
\] 
always exists. It suffices to show $A=0$. We prove this by a contradiction. If $A>0$, then we would have

\paragraph{Zeros are dense} We claim that there exist constants $C_0>0$ and $y_0\in (0,1)$, so that for any $y'\in (y_0,1)$, the interval $I(y') = (y', y'+C_0(1-y')^\frac{p+2}{p+1})$ contains either a zero of $f(y)$ or a zero of $f'(y)$. In fact, if the interval contained neither a zero of $f(y)$ or $f'(y)$, then $f(y)$ and $f'(y)$ could never change their signs. As a result, we have 
 \[
  \int_{I(y')} |f'(y)| dy = |f(y') - f(y'+C_0(1-y')^\frac{p+2}{p+1})| \lesssim (1-y'^2)^{-\frac{2\beta+1}{p+1}}. 
 \]
 Here we use the upper bound given in \eqref{boundedness of f}. Please note that the implicit constants associated with $\lesssim$ here (and in the argument below) do not depend on $C_0$, but the inequalities hold only if $y'>y_0(f, C_0)$ is sufficiently large. Similarly we may use the identity 
 \[
  \frac{d}{dy} \left[(1-y^2)^{\beta+1} f'(y)\right] = -(1-y^2)^\beta P'(f(y)),
 \]
 and obtain 
 \[
  \left|\int_{I(y')} (1-y^2)^\beta P'(f(y)) dy \right| \lesssim 1 \Rightarrow \int_{I(y')} (1-y^2)^\beta |f(y)|^{p} dy \lesssim 1. 
 \]
 We may ignore the integral of $(1-y^2)^\beta f(y)$ because 
 \[
  \int_{I(y')} (1-y^2)^\beta |f(y)| dy \lesssim  C_0 (1-y'^2)^{\beta - \frac{2\beta+1}{p+1}+\frac{p+2}{p+1}} = C_0 (1-y'^2)^{\frac{p+3}{p+1}} \lesssim 1. 
 \]
 In summary we have
 \[
   \int_{I(y')}  \left[(1-y'^2)^{\frac{2\beta+1}{p+1}} |f'(y)| + (1-y^2)^\beta |f(y)|^{p}\right] dy \lesssim 1.
 \]
 On the other hand, our assumption on the limit of $\tilde{E}(t)$ guarantees that either $|f'(y)|\gtrsim (1-y^2)^{-\beta-1}$ or $|f(y)| \gtrsim  (1-y^2)^{-\frac{2\beta+1}{p+1}}$ holds. Thus the integrand above satisfies
 \[
  (1-y'^2)^{\frac{2\beta+1}{p+1}} |f'(y)| + (1-y^2)^\beta |f(y)|^{p} \gtrsim (1-y'^2)^{-\frac{p+2}{p+1}}. 
 \] 
This gives a contradiction when $C_0$ is sufficiently large since $|I(y')| = C_0 (1-y'^2)^{\frac{p+2}{p+1}}$.

\paragraph{Loss of energy} Since the zeros are dense when $y \rightarrow 1^-$, given any $y_0 \in (1/2,1)$ we may find two numbers $y_1, y_2 \in (y_0,1)$ so that $1-y_1 \approx 2(1-y_2)$ and $f(y_1) f'(y_1) = f(y_2) f'(y_2) = 0$. By the identity 
\[
 \frac{d}{dy} \left[(1-y^2)^{\beta+1} f'(y)f(y)\right] = (1-y^2)^{\beta+1} |f'(y)|^2  -(1-y^2)^\beta P'(f(y)) f(y),
\]
we have
\[
 \int_{y_1}^{y_2} (1-y^2)^{\beta+1} |f'(y)|^2 dy = \int_{y_1}^{y_2} (1-y^2)^\beta P'(f(y)) f(y) dy \leq \int_{y_1}^{y_2} (1-y^2)^\beta |f(y)|^{p+1} dy. 
\]
Since $(1-y^2)$ are comparable to each other for all $y \in [y_1,y_2]$, we have 
\[
 \int_{y_1}^{y_2} (1-y^2)^{2\beta+2} |f'(y)|^2 dy \lesssim \int_{y_1}^{y_2} (1-y^2)^{2\beta+1} |f(y)|^{p+1} dy.
\]
The implicit constants associated to $\lesssim$ here (and in the argument below) depend on nothing but $p$. Thus we have 
\begin{align*}
 (y_2-y_1)A & \leq \int_{y_1}^{y_2} \tilde{E}(y) dy\\
  & \leq \int_{y_1}^{y_2} \left[\frac{1}{2} (1-y^2)^{2\beta+2} |f'(y)|^2 + (1-y^2)^{2\beta+1} \left(C_p+\frac{1}{p+1}|f(y)|^{p+1}\right)\right] dy\\
 & \lesssim (y_2-y_1) (1-y_1^2)^{2\beta+1} + \int_{y_1}^{y_2} (1-y^2)^{2\beta+1} |f(y)|^{p+1} dy.
\end{align*}
Please note that $y_2-y_1 \approx 1-y_2 \approx (1-y_1)/2$ by our assumption on $y_1,y_2$. In addition, when $y_1, y_2$ are sufficiently close to $1$, we have $(1-y_1^2)^{2\beta+1} \ll A$. Thus 
\[
 \int_{y_1}^{y_2} (1-y^2)^{2\beta} |f(y)|^{p+1} dy \gtrsim A. 
\]
Combining this with \eqref{lower bound of P} and \eqref{energy decay rate}, we obtain 
\[
 \tilde{E}(y_1) - \tilde{E}(y_2) = \int_{y_1}^{y_2} 2(2\beta+1)y(1-y^2)^{2\beta} P(f(y)) dy \gtrsim A. 
\]
This can never happen when $y_1,y_2$ are sufficiently close to $1$ because of our assumption $\tilde{E}(y) \rightarrow A$. 

\end{document}